\documentclass[10pt]{amsart}
\usepackage{graphicx}
\usepackage{amssymb}
\usepackage{amsmath,amsfonts,amssymb,amsthm,amscd,latexsym,euscript,xypic}
\usepackage[all]{xy}
\usepackage{verbatim}
\usepackage{url}
\usepackage{bbm}

\newtheorem{theorem}{Theorem}
\newtheorem*{maintheorem*}{Main Theorem}
\newtheorem{lemma}[theorem]{Lemma}
\newtheorem{sublemma}{Lemma}[theorem]
\newtheorem{proposition}[theorem]{Proposition}

\newtheorem{corollary}[theorem]{Corollary}
\theoremstyle{definition}
\newtheorem{definition}[theorem]{Definition}
\newtheorem{question}[theorem]{Question}

\newcommand{\crit}{\text{crit}}

\newcommand{\Vopenka}{Vop\v{e}nka}
\newcommand{\VP}{\rm VP}
\newcommand{\gVP}{\rm gVP}

\newcommand{\la}{\langle}
\newcommand{\ra}{\rangle}
\newcommand{\ZFC}{{\rm ZFC}}
\newcommand\GCH{{\rm GCH}}
\newcommand\CC{{\rm CC}}

\newcommand{\lesseq}[1]{{\smallleq}#1}
\newcommand{\smallleq}{\mathrel{\mathchoice{\raise2pt\hbox{$\scriptstyle\leq$}}{\raise1pt\hbox{$\scriptstyle\leq$}}{\raise1pt\hbox{$\scriptscriptstyle\leq$}}{\scriptscriptstyle\leq}}}
\newcommand{\smalllt}{\mathrel{\mathchoice{\raise2pt\hbox{$\scriptstyle<$}}{\raise1pt\hbox{$\scriptstyle<$}}{\raise0pt\hbox{$\scriptscriptstyle<$}}{\scriptscriptstyle<}}}

\newcommand{\p}{\mathbb P}
\newcommand{\q}{\mathbb Q}

\newcommand{\restrict}{\upharpoonright}
\newcommand\of{\subseteq}
\newcommand{\GBC}{{\rm GBC}}
\newcommand{\KM}{{\rm KM}}
\newcommand{\Godel}{G\"odel}
\newcommand{\ORD}{{\rm ORD}}
\newcommand{\tail}{\text{tail}}
\newcommand{\sm}{\text{small}}

\def\<#1>{\left\langle#1\right\rangle}
\newcommand\intersect{\cap}
\newcommand\union{\cup}

\newcommand\Q{\mathbb{Q}}
\newcommand{\image}{\mathbin{\hbox{\tt\char'42}}}
\newcommand{\set}[1]{\{\,{#1}\,\}}

\title[A model of the generic Vop\v enka principle where $\ORD$ is not Mahlo]{A model of the generic Vop\v enka principle in which the ordinals are not Mahlo}
\author{Victoria Gitman}
\address[V. Gitman]{The City University of New York, CUNY Graduate Center, Mathematics Program, 365 Fifth Avenue, New York, NY 10016, USA}
\email{vgitman@nylogic.org}
\urladdr{http://boolesrings.org/victoriagitman}

\author{Joel David Hamkins}
 \address[J.~D.~Hamkins]
         {Mathematics, Philosophy, Computer Science, The Graduate Center of The City University of New York, 365 Fifth Avenue, New York, NY 10016 \& Mathematics, College of Staten Island of CUNY}
\email{jhamkins@gc.cuny.edu}
\urladdr{http://jdh.hamkins.org}
\thanks{The research of the second author has been supported by grant \#69573-00 47 from the CUNY Research Foundation. Commentary concerning this paper can be made at http://jdh.hamkins.org/generic-vopenka-ord-not-mahlo.}

\begin{document}

\maketitle

\begin{abstract}
The generic \Vopenka\ principle, we prove, is relatively consistent with the ordinals being non-Mahlo. Similarly, the generic \Vopenka\ scheme is relatively consistent with the ordinals being definably non-Mahlo. Indeed, the generic \Vopenka\ scheme is relatively consistent with the existence of a $\Delta_2$-definable class containing no regular cardinals. In such a model, there can be no $\Sigma_2$-reflecting cardinals and hence also no remarkable cardinals. This latter fact answers negatively a question of Bagaria, Gitman and Schindler.
\end{abstract}

\section{Introduction}

The \emph{\Vopenka\ principle} is the assertion that for every proper class of first-order structures in a fixed language, one of the structures embeds elementarily into another. This principle can be formalized as a single second-order statement in \Godel-Bernays set-theory $\GBC$, and it has a variety of useful equivalent characterizations. For example, the \Vopenka\ principle holds precisely when for every class $A$, the universe has an $A$-extendible cardinal, and it is also equivalent to the assertion that for every class $A$, there is a stationary proper class of $A$-extendible cardinals (see~\cite[theorem 6]{Hamkins:The-Vopenka-principle-is-inequivalent-to-but-conservative-over-the-Vopenka-scheme}). In particular, the \Vopenka\ principle implies that $\ORD$ is Mahlo: every class club contains a regular cardinal and indeed, an extendible cardinal and more.

To define these terms, recall that a cardinal $\kappa$ is \emph{extendible}, if for every $\lambda>\kappa$, there is an ordinal $\theta$ and an elementary embedding $j:V_\lambda\to V_\theta$ with critical point $\kappa$. It turns out that, in light of the Kunen inconsistency, this weak form of extendibility is equivalent to a stronger form, where one insists also that $\lambda<j(\kappa)$; but there is a subtle issue about this that will come up later in our treatment of the virtual forms of these axioms, where the virtual weak and virtual strong forms are no longer equivalent. Relativizing to a class parameter, a cardinal $\kappa$ is \emph{$A$-extendible} for a class $A$, if for every $\lambda>\kappa$, there is an elementary embedding
 $$j:\la V_\lambda, \in, A\cap V_\lambda\ra\to \la V_\theta,\in,A\cap V_\theta\ra$$
with critical point $\kappa$, and again one may equivalently insist also that $\lambda<j(\kappa)$; see \cite[definition~6.7]{SolovayReinhardtKanamori1978:Strong-axioms-of-infinity-and-elementary-embeddings}. Every such $A$-extendible cardinal is therefore extendible and hence inaccessible, measurable, supercompact and more. These are amongst the largest large cardinals.

In the first-order \ZFC\ context, set theorists commonly consider a first-order version of the \Vopenka\ principle, which we call the \emph{\Vopenka\ scheme}, the scheme making the \Vopenka\ assertion of each definable class separately, allowing parameters.\footnote{Henceforth in this article, when we say `definable class,' we shall always mean that parameters are allowed, unless otherwise specifically mentioned.} That is, the \Vopenka\ scheme asserts, of every formula $\varphi$, that for any parameter $p$, if $\{\,x\mid \varphi(x,p)\,\}$ is a proper class of first-order structures in a common language, then one of those structures elementarily embeds into another.

The \Vopenka\ scheme is naturally stratified by the assertions $\VP(\Sigma_n)$, for the particular natural numbers $n$ in the meta-theory, where $\VP(\Sigma_n)$ makes the \Vopenka\ assertion for all $\Sigma_n$-definable classes. Using the definable $\Sigma_n$-truth predicate, each assertion $\VP(\Sigma_n)$ can be expressed as a single first-order statement in the language of set theory.

Hamkins~\cite{Hamkins:The-Vopenka-principle-is-inequivalent-to-but-conservative-over-the-Vopenka-scheme} proved that the \Vopenka\ principle is not provably equivalent to the \Vopenka\ scheme, if consistent, although they are equiconsistent over \GBC\ and furthermore, the \Vopenka\ principle is conservative over the \Vopenka\ scheme for first-order assertions. That is, over \GBC\ the two versions of the \Vopenka\ principle have exactly the same consequences in the first-order language of set theory.

In this article, we are concerned with the virtual forms of the \Vopenka\ principles. The main idea of virtualization, due Schindler, is to weaken statements asserting the existence of elementary embeddings between some \emph{set-sized} first-order structures to the assertion that such embeddings can be found in a forcing extension of the universe.  Schindler's remarkable cardinals, for example, instantiate the virtualized form of supercompactness via the Magidor characterization of supercompactness \cite{Schindler:RemarkableCardinals}. This virtualization program has now been undertaken with various large cardinals, leading to fruitful new insights (see~\cite{GitmanSchindler:virtualCardinals,BagariaGitmanSchindler:VopenkaPrinciple}). The notion of virtual large cardinals differs significantly from the conceptually related notion of generic large cardinals that has a much longer history (see, for instance, \cite{Foreman:IdealsGenericElementaryEmbeddings}). A generic version of a large cardinal notion asserts that a forcing extension $V[G]$ has embeddings $j:V\to M$ of the type characterizing the original notion, so that the target model $M$ is (in all interesting cases) not contained in $V$. Generic large cardinals generally have consistency strength in the neighborhood of their actual counterparts, but can themselves be very small, as for example $\omega_1$. In contrast, the embeddings witnessing virtual large cardinals are between set-sized structures from $V$. The virtual large cardinals are actual large cardinals, usually at least weakly compact, but compatible with $V=L$, and consequently much weaker than their actual counterparts \cite{GitmanSchindler:virtualCardinals}.

Carrying out the virtualization idea with the \Vopenka\ principles, we define the \emph{generic \Vopenka\ principle} to be the second-order assertion in \GBC\ that for every proper class of first-order structures in a common language, one of the structures admits, in some forcing extension of the universe, an elementary embedding into another. That is, the structures themselves are in the class in the ground model, but you may have to go to the forcing extension in order to find the elementary embedding.

Similarly, the \emph{generic \Vopenka\ scheme}, introduced in~\cite{BagariaGitmanSchindler:VopenkaPrinciple}, is the assertion (in \ZFC\ or \GBC) that for every first-order definable proper class of first-order structures in a common language, one of the structures admits, in some forcing extension, an elementary embedding into another.

On the basis of their work in~\cite{BagariaGitmanSchindler:VopenkaPrinciple}, Bagaria, Gitman and Schindler had asked the following question:

\begin{question}\label{Question:gVS-implies-remarkables?}
  If the generic \Vopenka\ scheme holds, then must there be a proper class of remarkable cardinals?
\end{question}

There seemed good reason to expect an affirmative answer, even assuming only $\gVP(\Sigma_2)$, based on strong analogies with the non-generic case. Specifically, in the non-generic context Bagaria had proved that $\VP(\Sigma_2)$ was equivalent to the existence of a proper class of supercompact cardinals, while in the virtual context, Bagaria, Gitman and Schindler proved that the generic form $\gVP(\Sigma_2)$ was equiconsistent with a proper class of remarkable cardinals, the virtual form of supercompactness. Similarly, higher up, in the non-generic context Bagaria had proved that $\VP(\Sigma_{n+2})$ is equivalent to the existence of a proper class of $C^{(n)}$-extendible cardinals, while in the virtual context, Bagaria, Gitman and Schindler proved that the generic form $\gVP(\Sigma_{n+2})$ is equiconsistent with a proper class of virtually $C^{(n)}$-extendible cardinals.

But further, they achieved direct implications, with an interesting bifurcation feature that specifically suggested an affirmative answer to question \ref{Question:gVS-implies-remarkables?}. Namely, what they showed at the $\Sigma_2$-level is that if there is a proper class of remarkable cardinals, then $\gVP(\Sigma_2)$ holds, and conversely if $\gVP(\Sigma_2)$ holds, then there is either a proper class of remarkable cardinals or a proper class of virtually rank-into-rank cardinals. And similarly, higher up, if there is a proper class of virtually $C^{(n)}$-extendible cardinals, then $\gVP(\Sigma_{n+2})$ holds, and conversely, if $\gVP(\Sigma_{n+2})$ holds, then either there is a proper class of virtually $C^{(n)}$-extendible cardinals or there is a proper class of virtually rank-into-rank cardinals. So in each case, the converse direction achieves a disjunction with the target cardinal and the virtually rank-into-rank cardinals. But since the consistency strength of the virtually rank-into-rank cardinals is strictly stronger than the generic \Vopenka\ principle itself, one can conclude on consistency-strength grounds that it isn't always relevant, and for this reason, it seemed natural to inquire whether this second possibility in the bifurcation could simply be removed. That is, it seemed natural to expect an affirmative answer to question~\ref{Question:gVS-implies-remarkables?}, even assuming only $\gVP(\Sigma_2)$, since such an answer would resolve the bifurcation issue and make a tighter analogy with the corresponding results in the non-generic/non-virtual case.

In this article, however, we shall answer the question negatively. The details of our argument seem to suggest that a robust analogy with the non-generic/non-virtual principles is achieved not with the virtual $C^{(n)}$-cardinals, but with a weakening of that property that drops the requirement that $\lambda<j(\kappa)$, as explained in theorem~\ref{th:Equivalencies-for-gVS}. Indeed, theorem~\ref{th:Equivalencies-for-gVS} seems to offer an illuminating resolution of the bifurcation aspect of the results we mentioned from~\cite{BagariaGitmanSchindler:VopenkaPrinciple}, because it provides outright virtual large-cardinal equivalents of the stratified generic \Vopenka\ principles. Because the resulting virtual large cardinals are not necessarily remarkable, however, our main theorem shows that it is relatively consistent with even the full generic \Vopenka\ principle that there are no $\Sigma_2$-reflecting cardinals and therefore no remarkable cardinals.

\begin{maintheorem*}\
\begin{enumerate}
  \item It is relatively consistent that \GBC\ and the generic \Vopenka\ principle holds, yet $\ORD$ is not Mahlo.
  \item It is relatively consistent that \ZFC\ and the generic \Vopenka\ scheme holds, yet $\ORD$ is not definably Mahlo, and not even $\Delta_2$-Mahlo. In such a model, there can be no $\Sigma_2$-reflecting cardinals and therefore also no remarkable cardinals.
\end{enumerate}
\end{maintheorem*}

These theorems are proved as theorems~\ref{th:ORDnotMahlo} and~\ref{th:DefinableORDnotMahlo}. The theorems are proved under the assumption that $0^\sharp$ exists, although this assumption can be weakened.

\section{Virtual embeddings}\label{sec:virtualEmbeddings}

In our main result, we shall make use of some absoluteness properties concerning virtual embeddings, and so let us review those ideas now. The following folklore results, which have appeared in a number of articles involving virtual large cardinals (possibly earliest in \cite{schindler:remarkable2}), are central.

\begin{lemma}[Absoluteness lemma]\label{lem:absolutenessLemma}
Suppose that $M$ is a countable first-order structure and $j:M\to N$ is an elementary embedding. If $W$ is a transitive (set or class) model of (some sufficiently large fragment of) $\ZFC$ such that $M$ is countable in $W$ and $N\in W$, then for any finite subset of $M$, the model $W$ has an elementary embedding $j^*:M\to N$, which agrees with $j$ on that subset. Moreover, if both $M$ and $N$ are transitive $\in$-structures and $j$ has a critical point, we can additionally assume that $\crit(j^*)=\crit(j)$.
\end{lemma}

\noindent The proof is an elementary tree-of-attempts argument and can be found in~\cite{GitmanSchindler:virtualCardinals}. As a consequence, we can say a little something more about which kind of forcing extensions one needs to look in to find the embeddings: if there is an elementary embedding $j:M\to N$ in some forcing extension, then there is one in any forcing extension in which $M$ has become countable.

\begin{lemma}
If $M$ and $N$ are first-order structures in a common language and there is an elementary embedding $j:M\to N$ in some set-forcing extension, then there is such an embedding $j^*:M\to N$ in any forcing extension in which $M$ has become countable. Further, one can arrange that $j^*$ agrees with $j$ on any prescribed finite set of values and that, if appropriate, $j$ and $j^*$ have the same critical point.
\end{lemma}

To prove the lemma, note that if a forcing extension $V[G]$ has an elementary embedding $j:M\to N$, then we can go to a further extension $V[G][H]$ by any forcing collapsing $M$ to become countable, and apply the absoluteness lemma in that extension, with the class $W$ being the desired extension of $V$ in which $M$ has become countable. So there is such a $j^*$ in that extension with the required similarities to $j$.

There is also an interesting game-theoretic characterization of when virtual embeddings exist, which makes no reference to forcing. Specifically, given first-order structures $M$ and $N$ in a common language, we define the associated two-player game $G(M,N)$ in which player I plays elements from $M$ and player II plays corresponding elements from $N$. Player II wins if at every stage of play, the moves constitute a finite partial isomorphism of $M$ to $N$. In other words, the type of the first $n$ moves of player I in $M$ is equal to the type of the first $n$ moves of player II in $N$. This game $G(M,N)$ is closed for player II and therefore determined by the Gale-Stewart theorem. So one of the players has a winning strategy. The characterization is provided by the following lemma.

\begin{lemma}[\cite{BagariaGitmanSchindler:VopenkaPrinciple}]
Suppose $M$ and $N$ are first-order structures. The following are equivalent.
\begin{enumerate}
\item In some set-forcing extension there is an elementary embedding $j:M\to N$.
\item Player II has a winning strategy in the game $G(M,N)$.
\end{enumerate}
\end{lemma}

One can prove this lemma simply by observing that the winner of any open game is absolute, since the recursive definition of the ordinal game values for the various positions in the game tree is defined identically in the two models. Since player II can clearly win in the forcing extension, where there is an actual elementary embedding, it follows that she must also have a winning strategy in the ground model.

Let us illustrate the easy power of these absoluteness results with the following application.

\begin{proposition}\label{ob:zero-sharp-implies-gVP-in-L} Assume $0^\sharp$ exists. Then:
 \begin{enumerate}
   \item The constructible universe $L$, equipped with only its definable classes, is a model of the generic \Vopenka\ principle.
   \item In $L$ there are numerous virtual rank-into-rank embeddings $j:V_\theta^L\to V_\theta^L$, where $\theta$ is far above the supremum of the critical sequence.
 \end{enumerate}
\end{proposition}

\begin{proof}
For statement (1), consider any class $A$ that is definable (from parameters) in $L$. Let $j:L\to L$ be an indiscernibility embedding with critical point $\kappa$ above rank of the parameters used to define $A$, which must then take $A$ to $A$. For any $\theta$ above $\kappa$ that is closed under $j$, it follows that $j\restrict L_\theta:\<L_\theta,\in,A\intersect L_\theta>\to\<L_\theta,\in,A\intersect L_\theta>$ is an elementary embedding with critical point $\kappa$. It follows by the previous results that there is a virtual embedding like that in a forcing extension of $L$, and so $\kappa$ is virtually $A$-extendible in $L$. Thus, the generic \Vopenka\ principle holds there.

For statement (2), similarly, let $j:L\to L$ shift a sequence of Silver indiscernibles down low and fix all other Silver indiscernibles. If $\lambda<\theta$ are Silver indiscernibles above the critical sequence, then $j\restrict L_\theta:L_\theta\to L_\theta$ is an elementary embedding in $V$ with fixed point $\lambda$ above the critical sequence. By the absoluteness lemma, therefore, there must be such embeddings in a forcing extension of $L$.
\end{proof}

The proof shows that every Silver indiscernible is virtually $A$-extendible in $L$ for every definable class $A$, and furthermore, is the critical point of virtual rank-into-rank embeddings with targets as high as desired and fixed points as high above the critical sequence as desired.

One doesn't need the full strength of $0^\sharp$, however, to get a model of the generic \Vopenka\ principle or to get virtual rank-into-rank embeddings, and the argument above shows that $0^\sharp$ has a strictly higher consistency strength than a virtual rank-into-rank cardinal, since one can simply chop off the universe at any Silver indiscernible and reflect these assertions, gaining a transitive model of the latter theories. For example, if $\kappa$ is virtually rank-into-rank in $L$, and $\theta>\kappa$ is a Silver indiscernible, then $L_\theta$ is a transitive model of $\ZFC$ with a virtually rank-into-rank cardinal. So the consistency strength of $0^\sharp$ is strictly stronger than necessary.

\section{Large cardinal characterizations of the generic \Vopenka\ principle and scheme}\label{sec:LC-characterizations}

Although our main theorem will use only the direct definition of the generic \Vopenka\ principle, let us sketch a richer background context for this principle by providing a large cardinal characterization of it. When working with the second-order generic \Vopenka\ principle our background theory is assumed to be $\GBC$.

\begin{definition}\rm
  A cardinal $\kappa$ is (weakly) \emph{virtually $A$-extendible}, for a class $A$, if for every ordinal $\lambda>\kappa$ there is an ordinal $\theta$ such that in a set-forcing extension, there is an elementary embedding
   $$j:\<V_\lambda,\in,A\intersect V_\lambda>\to\<V_\theta,\in,A\intersect V_\theta>,$$
  with critical point $\kappa$.
\end{definition}

In contrast, we define that $\kappa$ is (strongly) \emph{virtually $A$-extendible}, if we may also insist that $\lambda<j(\kappa)$ in the embedding mentioned above. Although in the non-virtual context mentioned in the introduction of this article, the weak and strong forms of $A$-extendibility coincide, nevertheless it turns out that as a consequence of the main theorem, the weak and strong forms of \emph{virtual} $A$-extendibility are not the same. Furthermore, our main theorem shows that it turns out to be the weak form that is relevant for the generic \Vopenka\ principle, and consequently, our article is concerned principally with the weak form of virtual $A$-extendibility.

\goodbreak
\begin{theorem}\label{th:Equivalencies-for-gVP}
The generic \Vopenka\ principle holds if and only if for every class $A$, there are a proper class of (weakly) virtually $A$-extendible cardinals.
\end{theorem}

\begin{proof}
We basically follow the argument of~\cite[theorem 6]{Hamkins:The-Vopenka-principle-is-inequivalent-to-but-conservative-over-the-Vopenka-scheme}, which gives the non-generic/non-virtual analogue of this characterization, except that (i) the embeddings here all live in a forcing extension; (ii) there is no requirement here that $\lambda<j(\kappa)$ for the embeddings and so this argument uses only the weak form of virtual extendibility; and (iii) we do not get here a stationary proper class of virtually $A$-extendible cardinals (and we cannot in light of corollary~\ref{Corollary.GVP-with-weak-but-not-strong}).

For the forward implication, assume that the generic \Vopenka\ principle holds and fix some class $A$ and an ordinal $\gamma$. We will argue that there is a (weakly) virtually $A$-extendible cardinal above $\gamma$. We claim first that for all sufficiently large $\lambda$, there is an ordinal $\theta>\lambda$ and a virtual elementary embedding $j:\<V_\lambda,\in,A\intersect V_\lambda>\to \<V_\theta,\in,A\intersect V_\theta>$ with critical point above $\gamma$. If not, let $\mathcal{M}$ be the proper class of all structures $\<V_\lambda,\in,A\intersect V_\lambda,\check\alpha>_{\alpha\leq\gamma}$, where we have added a constant symbol $\check\alpha$ for every ordinal $\alpha\leq\gamma$, for which there is no $\theta$ with the desired virtual embedding. Adding the constants is equivalent to requiring the critical point above $\gamma$. By the generic \Vopenka\ principle, there is an elementary embedding in some forcing extension between two of these structures $j:\<V_\lambda,\in,A\intersect V_\lambda,\check\alpha>_{\alpha\leq\gamma}\to\<V_\theta,\in,A\intersect V_\theta,\check\alpha>_{\alpha\leq\gamma}$, with $\lambda<\theta$, contrary to the inclusion of the former structure in $\mathcal{M}$, thereby establishing our claim. So we may fix an ordinal $\lambda_0$ such that for all $\lambda\geq\lambda_0$, there is an ordinal $\theta>\lambda$ and a virtual elementary embedding $j:\<V_\lambda,\in,A\intersect V_\lambda>\to\<V_\theta,\in,A\intersect V_\theta>$ with critical point above $\gamma$. For singular $\lambda$, we may assume without loss that $j$ has a critical point below $\lambda$, by considering $j\restrict V_\lambda$ for an embedding $j$ on $V_{\lambda+1}$, which must move $\lambda$, but cannot have $\lambda$ as its critical point. So we have a critical point $\kappa$ above $\gamma$ and less than $\lambda$, although different $\lambda$ could have different such critical points. Nevertheless, the map $\lambda\mapsto\kappa$, choosing the smallest such $\kappa$ that can be forced to be the critical point of such an embedding, is a definable pressing-down function. It follows that there is an unbounded class of $\lambda$ all giving rise to the same cardinal $\kappa$.\footnote{This weak version of the class Fodor's lemma, where one wants merely that the function is constant on an unbounded class, is easily provable in \GBC. The full class Fodor's lemma, in contrast, where the function is constant on a sationary class, is not provable even in \KM, but it is provable if one assumes the class choice principle \CC. See \cite{GitmanHamkinsKaragila:KM-set-theory-does-not-prove-the-class-Fodor-theorem}.} Thus, this constant value $\kappa$ is the critical point of virtual elementary embeddings $j:\<V_\lambda,\in,A\intersect V_\lambda>\to\<V_\theta,\in,A\intersect V_\theta>$ for unboundedly many ordinals $\lambda$. By restricting these embeddings, it follows that $\kappa$ is the critical point of such virtual embeddings for every $\lambda>\kappa$, and so $\kappa$ is a (weakly) virtually $A$-extendible cardinal above $\gamma$.

Conversely, suppose that every class $A$ has a proper class of (weakly) virtually $A$-extendible cardinals, and suppose that $\mathcal{M}$ is a class of first-order structures in a common language. Let $\kappa$ be a virtual $\mathcal{M}$-extendible cardinal above the size of the common language. Let $\lambda$ be the $\kappa^{th}$ ordinal above $\kappa$ for which there is an element $M\in\mathcal{M}$ of rank $\lambda$. Since $\kappa$ is virtually $\mathcal{M}$-extendible, there is an ordinal $\theta$ and a virtual elementary embedding $j:\<V_{\lambda+1},\in,\mathcal{M}\intersect V_{\lambda+1}>\to \<V_{\theta+1},\in,\mathcal{M}\intersect V_{\theta+1}>$ with critical point $\kappa$. By elementarity, it follows that $j(M)\in \mathcal{M}$ has rank $j(\lambda)$, which is the $j(\kappa)^{th}$ ordinal above $j(\kappa)$ that is the rank of an element of $\mathcal{M}$. In particular, $\lambda<j(\lambda)$ and consequently $M\neq j(M)$. Meanwhile, since the language is fixed pointwise by $j$, it follows that $j\restrict M:M\to j(M)$ is an elementary embedding between distinct elements of $\mathcal{M}$, thus verifying this instance of the generic \Vopenka\ principle.
\end{proof}

Essentially identical arguments establish the scheme-theoretic version, where we assume that the classes $A$ and $\mathcal{M}$ are definable.

\begin{theorem}\label{Theorem.Equivalences-gVS}
 The generic \Vopenka\ scheme is equivalent over \ZFC\ to the scheme asserting of every definable class $A$ that there is a proper class of virtually $A$-extendible cardinals.
\end{theorem}

As we have mentioned, there is a subtle but critical difference from the corresponding results for the non-generic non-virtual forms~\cite[theorems 6,7]{Hamkins:The-Vopenka-principle-is-inequivalent-to-but-conservative-over-the-Vopenka-scheme}. Namely, in the non-generic non-virtual forms, the \Vopenka\ principle is equivalent both to the assertion merely that every class $A$ has at least one $A$-extendible cardinal and also to the assertion that every class $A$ has a stationary proper class of $A$-extendible cardinals; and similarly the \Vopenka\ scheme is equivalent to the corresponding assertions about definable classes $A$. Those implications, however, rely on the strong form of $A$-extendibility, and that is fine because as we have mentioned, in the non-generic context the weak and strong forms of extendibility are equivalent. In our generic/virtual context here, however, the equivalence breaks down, and the generic \Vopenka\ principle entitles one only to the weak form of virtual extendibility, and not the strong form. Indeed, in corollary \ref{Corollary.GVP-with-weak-but-not-strong} we prove that it is relatively consistent with \GBC\ that every class $A$ admits a proper class of weakly virtually $A$-extendible cardinals, but no class $A$ admits even a single strongly virtually $A$-extendible cardinal.

The equivalence of the weak and strong forms of $A$-extendibility in the non-generic context relies fundamentally on an appeal to the Kunen inconsistency, as in~\cite[theorem~6]{Hamkins:The-Vopenka-principle-is-inequivalent-to-but-conservative-over-the-Vopenka-scheme}. In the generic/virtual case, however, there is no virtual form of the Kunen inconsistency, for one can have virtual Reinhardt cardinals, embeddings of the form $j:V_{\lambda+2}\to V_{\lambda+2}$, but where the embedding is added by forcing, as in observation~\ref{ob:zero-sharp-implies-gVP-in-L}, and this prevents one from making the analogue of the argument leading to $\lambda<j(\kappa)$ and hence from the weak form to the strong form or from individual $A$-extendible cardinals to stationary proper classes of them. The main result of this paper shows that this issue is inherent, since we prove that it is relatively consistent with the generic \Vopenka\ principle that $\ORD$ is not Mahlo. In such a model, there can be no stationary proper class of virtually extendible or virtually $A$-extendible cardinals, since there is not even a stationary proper class of regular cardinals.

By paying careful attention to the precise complexity of the definitions of the various classes, we may next present a stratified version of theorem~\ref{Theorem.Equivalences-gVS}. Let us say that a cardinal $\kappa$ is \emph{$(\Sigma_n)$-extendible}, if it is $A$-extendible, where $A$ is the $\Sigma_n$-truth predicate. (Kindly note the difference in notation from the $\Sigma_n$-extendible cardinals, used in~\cite{BagariaHamkinsTsaprounisUsuba2016:SuperstrongAndOtherLargeCardinalsAreNeverLaverIndestructible} with a different meaning.) Similarly, we say that $\kappa$ is weakly or strongly \emph{virtually $(\Sigma_n)$-extendible}, if it is respectively weakly or strongly virtually $A$-extendible for that class.

\begin{theorem}\label{th:Equivalencies-for-gVS}
For $n\geq 1$, the following are equivalent as schemes over \ZFC.
  \begin{enumerate}
    \item The generic \Vopenka\ scheme holds for $\Pi_{n+1}$-definable classes.
    \item The generic \Vopenka\ scheme holds for $\Sigma_{n+2}$-definable classes.
    \item For every $\Sigma_n$-definable class $A$, there is a proper class of (weakly) virtually $A$-extendible cardinals.
    \item There is a proper class of (weakly) virtually $(\Sigma_n)$-extendible cardinals.
    \item There is a proper class of cardinals $\kappa$, such that for every $\Sigma_n$-correct cardinal $\lambda$ above $\kappa$, there is a $\Sigma_n$-correct cardinal $\theta>\lambda$ and a virtual elementary embedding $j:V_\lambda\to V_\theta$ with critical point $\kappa$.
  \end{enumerate}
\end{theorem}

\begin{proof}
We should like to emphasize that for none of the embeddings here do we insist that $\lambda<j(\kappa)$; we are using the weak forms only.

($1\to 2$) Assume that the generic \Vopenka\ scheme holds for $\Pi_{n+1}$-definable classes and that $\mathcal{M}$ is a $\Sigma_{n+2}$ definable class of structures, defined by $x\in\mathcal{M}\iff \exists y\, \psi(x,y,p)$, where $\psi$ is $\Pi_{n+1}$ and $p$ is a fixed parameter. For each structure $x\in\mathcal{M}$, let $Y_x$ be the set of minimal-rank $y$ witnessing that $x\in\mathcal{M}$ by the defining property, and form a new structure $x^+$ by adding $Y_x$ to $x$ as a new point, if necessary, and interpreting it as a new constant symbol. Let $\mathcal{N}$ be the class of structures $x^+$ obtained in this way. This class is $\Pi_{n+1}$-definable and any embedding of structures in $\mathcal{N}$ gives rise to an embedding of the corresponding structures in $\mathcal{M}$, as desired.

($2\to 3$) Argue as in theorem~\ref{th:Equivalencies-for-gVP}. Assume that $A$ is a $\Sigma_n$-definable class, and fix any $\gamma$. First, the collection of structures $\<V_\lambda,\in,A\intersect V_\lambda,\check\alpha>_{\alpha\leq\gamma}$ that have no virtual embedding to some $\<V_\theta,\in,A\intersect V_\theta,\check\alpha>_{\alpha\leq\gamma}$ is $\Pi_{n+1}$-definable, and so we get a virtual embedding from one of them to another, with critical point above $\gamma$. So for all sufficiently large $\lambda$, there are such embeddings with some critical point, and we may apply the weak class Fodor lemma to find a single $\kappa$ that works unboundedly often.

($3\to 4$) This is immediate, since the $\Sigma_n$-truth predicate is $\Sigma_n$-definable.

($4\to 5$) If $\kappa$ is (weakly) virtually $(\Sigma_n)$-extendible and $\lambda$ is $\Sigma_n$-correct, then we get $j:\<V_\lambda,\in,A\intersect V_\lambda>\to \<V_\theta,\in,A\intersect V_\theta>$ with critical point $\kappa$, where $A$ is the $\Sigma_n$-truth predicate. Since $V_\lambda$ can verify that $A\intersect V_\lambda$ agrees with $\Sigma_n$-truth in $V_\lambda$, this will also be true for $V_\theta$, and so $\theta$ must also be $\Sigma_n$-correct, as desired.

($5\to 1$) Suppose that $\mathcal{M}$ is a $\Pi_{n+1}$-definable class of first-order structures in a common language, defined so that $x\in\mathcal{M}\iff \forall z\, \varphi(x,z,a)$, where $\varphi$ has complexity $\Sigma_n$. Let $\kappa$ be as in statement (5) and larger than the size of the language used for structures in $\mathcal{M}$. Let $m$ be much larger than $n$ and let $\lambda>\kappa$ be any $\Sigma_{m}$-correct ordinal. So it is also $\Sigma_n$-correct. By our assumption on $\kappa$, there is a virtual elementary embedding $j:V_\lambda\to V_\theta$ with critical point $\kappa$, where $\theta$ is $\Sigma_n$-correct. Let $M\in\mathcal{M}$ be any structure with rank amongst the $\kappa^{th}$ rank to occur for structures in $\mathcal{M}$. This is observed correctly in $V_\lambda$. By the elementarity of the embedding, $V_\theta$ thinks that $j(M)$ is in $\mathcal{M}^{V_\theta}$, although $V_\theta$ may be wrong about this class. Since $V_\theta$ is $\Sigma_n$-correct, however, it is correct about $\varphi$, and from this it follows that $\mathcal{M}\intersect V_\theta\of\mathcal{M}^{V_\theta}$. In particular, $V_\theta$ knows $M\in\mathcal{M}^{V_\theta}$. Meanwhile, since $j(M)$ has the $j(\kappa)^{th}$ rank, it is not hard to see that $M\neq j(M)$, and since the critical point of $j$ is above the language of $M$, it follows that $j\restrict M:M\to j(M)$ is an elementary embedding. By the absoluteness lemma~\ref{lem:absolutenessLemma}, there will be such an embedding in any forcing extension collapsing $M$ to be countable. So $V_\theta$ thinks that there is a virtual embedding between two elements of $\mathcal{M}^{V_\theta}$. By the elementarity of $j$, it follows that $V_\lambda$ must also think this is true. But by the choice of $\lambda$, we know that $V_\lambda$ is right about this. So we have verified this instance of the generic \Vopenka\ scheme.
\end{proof}

The equivalencies proved in theorem~\ref{th:Equivalencies-for-gVS} can be seen as uniformizing some of the results of~\cite{BagariaGitmanSchindler:VopenkaPrinciple}. To avoid the bifurcation there into the two cases, such as getting from $\gVP(\Sigma_2)$ either a proper class of remarkable cardinals or a proper class of virtual rank-into-rank cardinals, what was needed was to drop the requirement that $\lambda<j(\kappa)$ for the embeddings, and then one gets a pure equivalence as above. The generic \Vopenka\ principle simply doesn't entitle one to embeddings $j$ with the stronger property that $\lambda<j(\kappa)$.

\section{A model of the generic \Vopenka\ principle in which the ordinals are not Mahlo}\label{sec:GVP}

We shall now construct a model of \GBC\ plus the generic \Vopenka\ principle in which the ordinals are not Mahlo. The main idea will be to adapt Hamkins's observation that one can easily separate the \Vopenka\ principle from the \Vopenka\ scheme.

\begin{theorem}[\cite{Hamkins:The-Vopenka-principle-is-inequivalent-to-but-conservative-over-the-Vopenka-scheme}]
 If the \Vopenka\ scheme holds, then there is a class-forcing extension $V[C]$ where it continues to hold, yet, in which the \Vopenka\ principle fails and $\ORD$ is not Mahlo, although it remains definably Mahlo.
\end{theorem}

The class forcing $\p$ is simply the standard forcing to kill $\ORD$ is Mahlo, the forcing to add a class club $C$ avoiding the regular cardinals. Conditions in $\p$ are closed bounded sets containing no regular cardinals, ordered by end-extension. Over the \GBC\ model in which it is defined, this forcing is ${\leq}\gamma$-distributive for every ordinal $\gamma$, because in fact the collection of conditions that reach above $\gamma$ is a ${\leq}\gamma$-closed dense subclass of the forcing. Consequently, forcing with $\p$ over the model in which it is defined adds no new sets and preserves \GBC. For these reasons, this forcing is amongst the nicest kind of class forcing that there is: over any \GBC\ model, this forcing is definable; it has a definable forcing relation; it adds no new sets; and it preserves \GBC. Since the generic class $C$ itself witnesses that $\ORD$ is not Mahlo in the extension $V[C]$, it follows that the \Vopenka\ principle must fail there, but because the forcing adds no new sets, it preserves the \Vopenka\ scheme and consequently also the definable Mahloness of $\ORD$.

We shall adapt the method here in order to prove that the generic \Vopenka\ scheme is also relatively consistent with the non-Mahloness of $\ORD$.

\begin{theorem}\label{th:ORDnotMahlo}
Assume $0^{\#}$ exists in $V$. Then there is a class-forcing notion $\p$ definable in the constructible universe $L$, such that in any $L$-generic extension $L[C]$ by this forcing, \GBC\ and the generic \Vopenka\ principle hold, yet $\ORD$ is not Mahlo.
\end{theorem}

\begin{proof}
To begin, assume $0^\sharp$ exists, and let $\p$ be the class-forcing notion, as defined in $L$, to add a class club $C$ avoiding the regular cardinals. In $L$, this forcing is ${\leq}\gamma$-distributive for every ordinal $\gamma$, and as we mentioned earlier, using this forcing over $L$ adds no new sets; it has a definable forcing relation; and it preserves \GBC. We view $L[C]$ as a \GBC\ model having the classes that are definable from $C$, or in other words, definable in the structure $\<L,\in,C>$. Since $C$ is a class club containing no regular cardinals, it follows that $L[C]$ thinks that $\ORD$ is not Mahlo.

Our use of class forcing is a bit unusual here, because although we assume $0^\sharp$ exists, we do not force over $V$, but rather only over $L$. We make the $0^\sharp$ assumption only in order to establish a certain density property for the forcing $\p$ in $L$, in order to know that it will succeed when used to force over $L$. Indeed, while the forcing $\p$ is very nice for forcing over $L$, meanwhile it is much less nice to force with $\p$ over $V$---this will definitely destroy \GBC. The reason is that because of $0^\sharp$, the full model $V$ already has a class club of $L$-regular cardinals, but these two class clubs cannot intersect. So if $C$ is $V$-generic for $\p$, then in $V[C]$ we would be able to define a countable sequence cofinal in the ordinals, violating \GBC. This is not a problem for our argument, however, because we shall make no reference to forcing over $V$ and we shall never form the extension $V[C]$. Instead, our desired model is $L[C]$, which is a model of $\GBC$, whose first-order part (namely the collection of sets) is $L$.

What remains is to prove that the generic \Vopenka\ principle holds in $L[C]$. For this, we make a density argument in the following lemma scheme. Since $\p$ is definable in $L$, for any ordinal $\theta$, we may consider the analogue of the forcing $\p^{L_\theta}$ as defined inside $L_\theta$. A set of ordinals $c$ is $L_\theta$-generic for $\p^{L_\theta}$ if it meets all dense sets of $\p^{L_\theta}$ definable over $L_\theta$, in the sense that every such dense set contains an initial segment of $c$.

\begin{sublemma}\label{le:keyLemma1}
Consider any ordinal $\delta$ and suppose $n$ is a particular natural number of the meta-theory. Let $D_{\delta,n}$ be the collection of conditions $c\in\p$ for which there is an ordinal $\theta$ such that
\begin{enumerate}
    \item $L_\theta\prec_{\Sigma_n}L$,
    \item $c\intersect\theta$ is $L_\theta$-generic for $\p^{L_\theta}$, and
    \item in some forcing extension of $L$, there is an elementary embedding
    $$j:\la L_\theta,\in,c\intersect\theta\ra\to \la L_\theta,\in, c\intersect\theta\ra$$ with critical point above $\delta$.
\end{enumerate}
Then $D_{\delta,n}$ is a definable dense subclass of $\p$ in $L$.
\end{sublemma}

\begin{proof}
Fix any ordinal $\delta$ and any particular meta-theoretic natural number $n$ (the lemma is a scheme as $n$ changes). We want to show $D_{\delta,n}$ is dense in $\p$. Since the class $D_{\delta,n}$ gets smaller as $n$ increases, we may imagine without loss that $n$ is very large. Fix any condition $d\in\p$. We shall find $\bar c\in D_{\delta,n}$ extending $d$. Let $\kappa_0$ be any uncountable cardinal of $V$ above $\delta$ and the supremum of $d$. Next, let
 $$\kappa_0<\kappa_1<\cdots<\kappa_n<\cdots<\kappa_\omega<\kappa_{\omega+1}$$
be the next $\omega+2$ many successive Silver indiscernibles. Let $\theta$ be the least ordinal above $\kappa_\omega$ such that $L_\theta\prec L_{\kappa_{\omega+1}}$. Note that because $\kappa_{\omega+1}$ is a Silver indiscernible, this will imply $L_\theta\prec L$, and in particular, $L_\theta\prec_{\Sigma_n} L$. The ordinal $\theta$ has cofinality $\omega$ in $L$ and is a limit cardinal there, and so we may find an $\omega$-sequence of $L$-cardinals $$\theta_0<\theta_1<\cdots<\theta_n<\cdots$$ cofinal in $\theta$ in $L$. We shall now construct $c\subseteq\theta$ end-extending $d$ that is $L_\theta$-generic for $\p^{L_\theta}$.

For each $k\in\omega$, let $D_k$ be the intersection of all open dense subclasses of $\p^{L_\theta}$ that are $\Sigma_k$-definable in $L_\theta$ using parameters in $L_{\theta_k}$. That is, we limit both the complexity of the definition and the space of parameters. Since $\Sigma_k$-truth is definable, the model $L_\theta$ has uniform definable access to its $\Sigma_k$-definable classes. And since there are only $\theta_k$ many parameters involved, the model $L_\theta$ may therefore enumerate its $\Sigma_k$-definable classes with parameters in $L_{\theta_k}$ in a single $\theta_k$-sequence. Since the forcing $\p^{L_\theta}$ is ${\leq}\theta_k$-distributive, it follows that $D_k$ is an open dense subclass of $\p^{L_\theta}$.

The classes $D_k$ provide a countable collection of dense subclasses that suffice for $L_\theta$-genericity, since any subclass of $\p^{L_\theta}$ that is definable in $L_\theta$ will be contained in some $D_k$. By the usual diagonalization procedure, therefore, we may build a set $c\of\theta$ extending $d$ by successively extending it so as to meet each $D_k$ in turn. It follows that $c\of\theta$ is $L_\theta$-generic for $\p^{L_\theta}$. Since this diagonalization construction can be carried out inside $L$, we may find such $L_\theta$-generic sets $c$ in $L$. Henceforth, let $c$ be the $L$-least such $L_\theta$-generic set for $\p^{L_\theta}$ extending $d$. Note that $c$ is definable in $L$ from parameters $d$, $\kappa_\omega$, and $\kappa_{\omega+1}$.

Let $j:L\to L$ be the elementary embedding generated by a shift of the indiscernibles $\kappa_n$, so that $j(\kappa_n)=\kappa_{n+1}$ for $n\in\omega$, and fixing all other indiscernibles. Since we chose $\kappa_0$ to be an uncountable cardinal of $V$, it follows that it is a limit of smaller indiscernibles, which generate $L_{\kappa_0}$. It follows that the critical point of $j$ is precisely $\kappa_0$ and in particular, $j(d)=d$. Since $\kappa_\omega$ and $\kappa_{\omega+1}$ are fixed by $j$, it follows similarly that $j(\theta)=\theta$. Since $c$ was defined by those fixed points, it also follows that $j(c)=c$. Thus, the restriction of $j$ to $L_\theta$ gives an elementary embedding
 $$j:\la L_\theta,\in, c\ra\to \la L_\theta,\in,c\ra.$$
By the absoluteness lemma~\ref{lem:absolutenessLemma}, it follows that in some forcing extension of $L$, there is an elementary embedding $$j^*: \la L_\theta,\in,c\ra\to \la L_\theta,\in,c\ra$$ with critical point $\kappa_0$. Finally, let $\bar c=c\union\{\theta\}$ be the closure of this set, and observe that this is a condition in $\p$ precisely because $\theta$ is singular in $L$. We have therefore verified all the necessary requirements to conclude that $\bar c\in D_{\delta,n}$, and since $\bar c$ extends $d$, we have therefore proved that this class is dense, establishing lemma~\ref{le:keyLemma1}.
\end{proof}

Using the lemma, we shall now complete the proof of theorem~\ref{th:ORDnotMahlo}. Fix any proper class $\mathcal M$ in $L[C]$ of first-order structures in a common language $\mathcal L$. Since we have included only the $C$-definable classes in $L[C]$, we may assume that $\mathcal{M}$ is defined by some $\Sigma_m$-formula $\psi(x,a,C)$ with class parameter $C$ and set parameter $a\in L$. Let $\delta$ be above the rank of the language $\mathcal L$ and the parameter $a$, and let $n$ be much larger than $m$. By the lemma, there is some ordinal $\theta$ such that the corresponding initial segment of $C$ is in $D_{\delta,n}$. So there is, in some forcing extension of $L$, an elementary embedding
 $$j:\la L_\theta,\in, C\cap \theta\ra\to \la L_\theta,\in,C\cap\theta\ra$$
with critical point $\kappa$ above $\delta$.

We claim next that
    $\la L_\theta,\in,C\cap\theta\ra\prec_{\Sigma_m} \la L,\in,C\ra$.
To see this, suppose that $\la L_\theta,\in,C\cap \theta\ra\models\varphi(b)$ for some $\Sigma_{m}$-formula $\varphi(x)$. This must be forced over $L_\theta$ by some condition $c$, an initial segment of $C\intersect\theta$. By the choice of $n$ as much larger than $m$, what we had meant was that it should be large enough to express this forcing relation, and so since $L_\theta\prec_{\Sigma_n}L$ it follows that $c$ forces $\varphi(b)$ also over $L$. So $\<L,\in,C>\models\varphi(b)$, as desired.

It follows that the definition of the class $\mathcal{M}$ is absolute to $\<L_\theta,\in,C\intersect\theta>$. Let $M$ be the $\kappa^{th}$ element of $\mathcal{M}$ in the $L$-order. It follows that $j(M)$ is the $j(\kappa)^{th}$ element, which is identified correctly by $L_\theta$, and consequently $M\neq j(M)$. The restriction $j\restrict M:M\to j(M)$ is an elementary embedding between distinct elements of $\mathcal{M}$, and so we have verified this instance of the generic \Vopenka\ principle in $L[C]$, as desired, completing the proof of theorem~\ref{th:ORDnotMahlo}.
\end{proof}

It was convenient to use $0^\sharp$ in the previous argument, in order to know that $D_{\delta,n}$ is dense in $\p$, but it is natural to inquire whether this use can be weakened or eliminated. We shall discuss this in section~\ref{sec:Eliminate-zero-sharp}.

\section{A model of the generic \Vopenka\ scheme in which the ordinals are not $\Delta_2$-Mahlo}\label{sec:GVS}

We should now like to sharpen the result of the previous section by performing further forcing so as to make the class club $C$ definable, in fact $\Delta_2$-definable, while preserving the generic \Vopenka\ scheme. The result will be a model of \GBC\ plus the generic \Vopenka\ scheme in which there are no $\Sigma_2$-reflecting cardinals (inaccessible $\Sigma_2$-correct cardinals). Consequently, there will also be no remarkable cardinals.

\begin{theorem}\label{th:DefinableORDnotMahlo}
Assume $0^\#$ exists in $V$. Then there is a definable class-forcing notion in $L$, such that in the corresponding $L$-generic extension, \GBC\ holds, the generic \Vopenka\ scheme holds, but $\ORD$ is not definably Mahlo. Indeed, in this model there is a $\Delta_2$-definable class club avoiding the regular cardinals.
\end{theorem}

\begin{proof}
The forcing will be a two-step iteration, although this forcing can also be viewed as a single-step class forcing. First, we force with $\p$ as in section~\ref{sec:GVP} to add a class club $C$ avoiding the regular cardinals of $L$. We proved in theorem~\ref{th:ORDnotMahlo}, under the assumption that $0^\sharp$ exists, an assumption we also have here, that the generic \Vopenka\ principle holds in $L[C]$, but the ordinals are not Mahlo there. Next, we force over $L[C]$ with the forcing $\q$ that codes the class $C$ into the continuum pattern. Specifically, $\q$ is the ${\ORD}$-length Easton product forcing, as defined in $L[C]$, which forces violations of the $\GCH$ exactly at the successor cardinals of the elements of $C$. This forcing is very mild in terms of class forcing. It is definable; it factors in the Easton manner at every element of $C$ into the product of small forcing and highly closed forcing; it is therefore progressively closed; and consequently, the forcing $\q$ preserves \GBC\ and has a definable forcing relation satisfying the forcing theorems.

The two-step iteration $\p*\dot\q$ can be viewed as a one-step non-iterative class forcing, simply by viewing the conditions as pairs $(d,q)$, where $d$ is a closed bounded set in $L$ avoiding the $L$-regular cardinals, and $q$ is a condition in the corresponding Easton-support product forcing to code $d$ into the \GCH\ pattern. This is dense in the iteration forcing, consisting of pairs $(d,\dot q)$ where $d\in\p$ and $\dot q$ is a $\p$-name for a condition in $\dot\q$, because the $\p$ forcing adds no new sets and so one may simply strengthen the first coordinate so as to decide the value of $\dot q$ in the second coordinate. Further, if $\dot q$ is any $\p$-name for a condition in $\dot\Q$, then there is a set-sized maximal antichain of conditions $d\in\p$ that decide the value of $\dot q$ as $\check q$ for some $q\in L$. Thus, any name for the iteration forcing can be transformed to a name for the non-iterative combined forcing we mentioned, and therefore the two versions of the forcing give rise to the same forcing extensions. So we needn't think of it as an iteration at all, and the combined forcing inherits the nice properties of $\p$ and $\q$ and will preserve \GBC\ and have a definable forcing relation satisfying the forcing theorem and so on.

Having added the generic class club $C$, let $G\subseteq \q$ be $L[C]$-generic, and we consider the forcing extension $L[C][G]$. This is a \GBC\ model, whose sets are those added by the $\Q$ forcing and whose classes are those definable from the generic classes $C$ and $G$. But actually, since ultimately we aim merely to construct a \ZFC\ model, let us use the notation $L[G]$ to refer only to the first-order part of this model, having only the sets and not the classes. That is, we take
 $$L[G]=\{\,\tau_G\mid \tau\text{ is a }\q\text{-name in }L\,\}$$
to consist of the interpretations via $G$ of the $\q$-names in $L$. This is a model of \ZFC, since it is the first-order part of a model of \GBC.

The coding forcing $\q$ ensures that the successor cardinals of the elements of $C$ can be identified in a $\Delta_2$ manner in the extension $L[G]$, and so the class $C$ is $\Delta_2$-definable in $L[G]$. Thus, in $L[G]$ we have a parameter-free $\Delta_2$-definable class club avoiding the regular cardinals. So $\ORD$ is definitely not Mahlo there, and not even $\Delta_2$-Mahlo.

It remains to argue that $L[G]$ satisfies the generic \Vopenka\ scheme. Let $\q_\theta$ be the factor of $\q$ consisting of the forcing only on the coordinates below $\theta$, that is, with conditions having support contained in $\theta$. Note that if $C\intersect\theta$ is unbounded in $\theta$, but $\theta$ is singular, then because of the nature of the Easton support, the forcing $\q$ will have conditions with support unbounded in $\theta$. In particular, $\q_\theta$ also has such conditions, and so we cannot view $\q_\theta$ as class forcing over the structure $\<L_\theta,\in,C\intersect\theta>$. That structure simply doesn't include all the conditions of $\q_\theta$. Nevertheless, since $C\intersect\theta$ is a set in $L$, it follows that $\q_\theta$ is a forcing notion in $L$. If $G\of\Q$ is $L$-generic, let us denote by $G_\theta$ the restriction of $G$ to $\q_\theta$.

\begin{sublemma}\label{le:keyLemma2}
 Suppose that $\delta$ is an ordinal and $n$ is a particular natural number of the meta-theory. Then there is an ordinal $\theta$ for which
  \begin{enumerate}
    \item $L_\theta\prec_{\Sigma_n}L$,
    \item $C\cap\theta$ is generic for dense subsets of $\p^{L_\theta}$ that are $\Sigma_n$-definable in $L_\theta$,
    \item $G_\theta\intersect \q^{L_\theta[C\cap\theta]}$ is generic for dense subsets of $\q^{L_\theta[C\intersect\theta]}$ that are $\Sigma_n$-definable in $L_\theta[C\intersect\theta]$, and
    \item in a forcing extension of $L[G]$, there is an elementary embedding $j:L_\theta[G_\theta]\to L_\theta[G_\theta]$ with critical point above $\delta$.
  \end{enumerate}
\end{sublemma}

\begin{proof}
This is a lemma scheme, taken as $n$ varies. But since the statement becomes harder as $n$ increases, we may assume without loss that $n$ is very large. Fix $n$ and any ordinal $\delta$. Let us explain a little further about what we mean in statements (3) and (4). The forcing notion $\q_\theta$ is the factor of $\q$ at coordinates up to $\theta$, but since $\theta$ is singular, this allows for conditions with unbounded support. The forcing notion $\q^{L_\theta[C\intersect\theta]}$, in contrast, is defined just as $\q$, except internally to $L_\theta[C\intersect\theta]$, which means that we now take only the bounded-support conditions. Since $G$ is fully $L[C]$-generic for $\q$, it follows easily that $G_\theta\of\q_\theta$ is fully $L[C\intersect\theta]$-generic for $\q_\theta$. What we are claiming in statement (3), however, is something a bit more, namely, that for this particular $\theta$, the bounded-support fragment of $G_\theta$, meaning $G_\theta\intersect\q^{L_\theta[C\intersect\theta]}$, is generic for dense subsets of $\q^{L_\theta[C\intersect\theta]}$ that are $\Sigma_n$-definable in $L_\theta[C\intersect\theta]$. In statement (4), the structure $L_\theta[G_\theta]$ can therefore be viewed either as the extension arising from the bounded-support forcing, or as $V_\theta^{L[G]}$, since it is not difficult to see that these are the same (for $n\geq 1$), since every new set is added by a stage.

To prove the lemma, since the statement of the lemma is expressible in the forcing language using the canonical names for the generic filters, it suffices to show that no condition forces the negation of the statement.\footnote{Note that if we had stated in the lemma that $L_\theta\prec L$, which is what we shall achieve in the proof, using the $0^\sharp$ assumption, then the statement of the lemma would not be expressible in $L$ in the forcing language, since $L$ cannot express the property $L_\theta\prec L$. It shall be enough for our purpose, however, to require only $L_\theta\prec_{\Sigma_n} L$ and state the lemma as a scheme over natural numbers $n$ in the meta-theory.} Fix any condition $(d,r)\in\p*\q$ in our class forcing.

We proceed at first as in lemma~\ref{le:keyLemma1}. Let $\kappa_0$ be an uncountable cardinal of $V$ above the ranks of $d$ and $r$ and also above $\delta$. Let
    $$\kappa_0<\kappa_1<\cdots<\kappa_n<\cdots<\kappa_\omega<\kappa_{\omega+1}$$
be the next $\omega+2$ many successive Silver indiscernibles.

Let $d^+=d\union\{\kappa_\omega+1\}$, which jumps over all the $\kappa_n$'s and is a condition in $\p$. To find $\theta$, suppose temporarily that $d^+$ agrees with $C$ and $r\in G$ (if this is not the case, replace them with generic filters that do have this property). In the extension $L[C][G]$, we have \ZFC\ in the language with predicates for $C$ and $G$. And $C$ is certainly generic for $\Sigma_n$-definable subclasses of $\p$ and $G$ is generic for dense subclasses of $\q$ that are $\Sigma_n$-definable in $L[C]$. By the reflection theorem, therefore, there must be an ordinal $\theta$ of cofinality $\omega$ reflecting these facts to $\theta$, so that $L_\theta\prec_{\Sigma_n} L$ and $C\intersect\theta$ is $\Sigma_n$-generic and $G\intersect\q^{L[C\intersect\theta]}$ is $\Sigma_n$-generic for $\q^{L_\theta[C\intersect\theta]}$ over $L_\theta[C\intersect\theta]$. So there must be a condition $(c,q)$ extending $(d^+,r)$ and forcing that $C\intersect\theta$ and $\dot G_\theta$ are like this. Note that no part of $c$ or $q$ above $\theta$ can matter for this property, and so we may assume $\sup(c)=\theta$ and $q\in \q_\theta$ as defined from $c$.

Now, throw away the previous actual $C$ and $G$, and let $\theta$ be the least ordinal with cofinality $\omega$ for which $\kappa_\omega<\theta<\kappa_{\omega+1}$ and there is a condition $(c,q)$ in the forcing $\p*\dot\q$ forcing the properties we have mentioned. And let $(c,q)$ be the $L$-least such condition extending $(d^+,r)$. So these conditions are definable from the parameters $(d,r)$, $\kappa_\omega$, and $\kappa_{\omega+1}$.

In $V$, let $j:L\to L$ be the elementary embedding generated by shifting the indiscernibles $\kappa_n$, so that $j(\kappa_n)=\kappa_{n+1}$ for all $n\in\omega$, but all other indiscernibles are fixed. Since $\kappa_0$ is an uncountable cardinal of $V$, it follows as before that the critical point of $j$ is exactly $\kappa_0$, and so $j$ fixes both $d$ and $r$. Since it also fixes $\kappa_\omega$ and $\kappa_{\omega+1}$, it follows by the definability considerations that it fixes $\theta$, $c$ and $q$.

Consider now just the forcing $\q_\theta$ defined by $c$ in $L$. We claim that in a suitable forcing extension of $V$, we will be able to find an $L$-generic filter $H_\theta\of\q_\theta$ with $q\in H_\theta$, such that $j$ lifts to
    $$j:L[H_\theta]\to L[H_\theta]$$
and such that $j(H_\theta)=H_\theta$. To begin with this, note that since $c$ is empty on the interval $[\kappa_0,\kappa_\omega)$, it follows that $\q_\theta$ is trivial on this interval. So the coding forcing $\q_\theta$ factors as very small forcing $\q^\sm_\theta$ to code $c$ below $\kappa_0$, followed by the coding forcing $\q^\tail_\theta$, which is above $\kappa_\omega$. So we have $$\q_\theta\cong \q^\tail_\theta\times \q^\sm_\theta,$$ where $\q^\tail_\theta$ is $\lesseq\kappa_\omega$-closed in $L$ and $\q^\sm_\theta$ has size less than $\kappa_0$ in $L$. We may similarly factor the condition $q$ as $q=q^\sm\times q^\tail$.

We claim that in order for a filter $H\of\q^\tail_\theta$ to be $L$-generic, it suffices for $H$ to meet all the open dense subsets of this forcing that are definable in $L$ using only indiscernible parameters not smaller than $\kappa_\omega$. To see this, suppose $D\subseteq \q_\theta^\tail$ is an arbitrary open dense subset of the forcing in $L$. Since every set in $L$ is definable from the Silver indiscernibles, there is a formula $\varphi$ and indiscernible parameters $\vec \kappa_\sm,\vec\kappa_\tail$, such that
    $$p\in D\ \leftrightarrow\ \varphi(p,\vec \kappa_\sm,\vec\kappa_\tail),$$
where $\vec\kappa_\sm$ are below $\kappa_\omega$ and $\vec\kappa_\tail$ are not.  Let
    $$D_{\vec\alpha}=\{p\in \q_\theta^\tail\mid\varphi(p,\vec\alpha,\vec\kappa_\tail)\},$$
where we allow arbitrary parameters $\vec\alpha$ in place of $\vec\kappa_\sm$. Let $\bar D$ be the intersection of all $D_{\vec\alpha}$ that happen to be open and dense in $\q_\theta^\tail$, ranging over all $\vec\alpha<\kappa_\omega$. Since the forcing $\q_\theta^\tail$ is ${\leq}\kappa_\omega$-closed, it follows that $\bar D$ is dense open. But furthermore, by its nature, $\bar D$ is definable from $\vec\kappa_\tail$, which are Silver indiscernibles not less than $\kappa_\omega$. If a filter meets $\bar D$, then it also meets $D$, since $\bar D\of D$.

Notice furthermore that if a condition $p\in\q_\theta^\tail$ is definable from indiscernible parameters other than the $\kappa_n$'s and a dense set $D$ is definable from non-$\kappa_n$ indiscernible parameters, then the $L$-least element in $D$ extending $p$ will also be definable from such parameters. Moreover, any such condition is a fixed point $j(p)=p$, since $j$ fixes all indiscernibles except the $\kappa_n$.

Let's go now to a forcing extension $V[K]$ by collapsing $\theta^+$ to become countable. We shall argue that in $V[K]$ we can construct the desired filter $H_\theta^\tail$. The extension $V[K]$ sees that there are only countably many subsets of $\q_\theta^\tail$ in $L$, and it can observe which of them are open and dense and definable in $L$ using only indiscernibles not less than $\kappa_\omega$. If $\<D_n\mid n<\omega>$ is the enumeration of these, then let $q_0$ be the $L$-least condition in $D_0$ extending $q^\tail$. Similarly, let $q_{n+1}$ be the $L$-least extension of $q_n$ in $D_{n+1}$. Thus, if $H^\tail_\theta$ is the filter generated by the conditions $q_n$, it will be $L$-generic for $\q_\theta^\tail$ and contain the condition $q^\tail$. Further, since each $q_n$ is definable from indiscernibles below $\kappa_\omega$ and above, it follows by the observation of the previous paragraph that $j(q_n)=q_n$ for each $n$. Thus, $j\image H^\tail_\theta=H^\tail_\theta$ and so we fulfill the lifting criterion for the embedding $j:L\to L$, which now lifts to $j:L[H^\tail_\theta]\to L[H^\tail_\theta]$, with $j(H^\tail_\theta)=H^\tail_\theta$.

And if $H^\sm_\theta$ is any further $L[H^\tail_\theta]$-generic filter for $\q_\theta^\sm$, then since this forcing is below the critical point of $j$ it also lifts easily to $j:L[H_\theta]\to L[H_\theta]$, where $H_\theta=H^\sm_\theta\times H^\tail_\theta\of\q_\theta$, with $j(H_\theta)=H_\theta$. We may assume $q^\sm\in H^\sm_\theta$ and therefore $q\in H_\theta$.
So in a forcing extension of $V$, we have found the kind of embedding we claimed.

It follows by the absoluteness lemma~\ref{lem:absolutenessLemma} that in a forcing extension of $L$, there is an $L$-generic filter $H_\theta\of\q_\theta$ containing the condition $q$ and an elementary embedding $j^*:L_\theta[H_\theta]\to L_\theta[H_\theta]$ with critical point above $\delta$. We also have $L_\theta\prec L$. Now we can extend $c$ and $H_\theta$ fully to $L$-generic filters $C$ and $H$ for the forcing $\p*\dot\q$. So finally, because $(c,q)\in C*H$, we have that $C\intersect\theta$ is $\Sigma_n$-generic for $\p^{L_\theta}$ over $L_\theta$ and $H_\theta\intersect\q^{L_\theta[C\cap\theta]}$ is $\Sigma_n$-generic over $L_\theta[C\cap\theta]$.

Since we have attained precisely the properties stated in the lemma, but for a generic filter containing the original condition $(d,r)$, our argument shows that $(d,r)$ could not have forced that this situation does not happen. And so lemma~\ref{le:keyLemma2} is proved.
\end{proof}

We now continue with the proof of theorem~\ref{th:DefinableORDnotMahlo}. In order to prove that the generic \Vopenka\ scheme holds in $L[G]$, suppose that $\mathcal M=\set{x\mid\psi(x,a)}$ is a definable class of first-order structures in a common language in $L[G]$, defined by the $\Sigma_m$-formula $\psi$ with parameter $a$. Let $n$ be much larger than $m$, and let $\delta$ be large enough so that the common language of the structures in $\mathcal{M}$ and the parameter $a$ have rank less than $\delta$. By lemma~\ref{le:keyLemma2}, there is an ordinal $\theta$ such that $L_\theta\prec_{\Sigma_n} L$, with $C\intersect\theta$ being $\Sigma_n$-generic for $\p^{L_\theta}$ over $L_\theta$ and $G\intersect\q^{L_\theta[C\intersect\theta]}$ being $\Sigma_n$-generic for $\q^{L_\theta[C\intersect\theta]}$ over $L_\theta[C\intersect\theta]$, and such that, in some forcing extension of $L$, there is an embedding $j:L_\theta[G_\theta]\to L_\theta[G_\theta]$ with critical point above $\delta$.

Since $n$ is much larger than $m$, it follows by the definability of the forcing relation that $\mathcal{M}$ is absolute to $L_\theta[G_\theta]$. Let $M$ be a structure in $\mathcal{M}$ of the $\kappa^{th}$ rank occurring in this class. So $j(M)$ is in $\mathcal{M}$ of the $j(\kappa)^{th}$ rank, and so $M\neq j(M)$. Since the common language of the structures in $\mathcal{M}$ is fixed by $j$, it follows that $j\restrict M:M\to j(M)$ is an elementary embedding. So we have witnessed the existence of a virtual elementary embedding between distinct structures in $\mathcal{M}$, and thereby verified this instance of the generic \Vopenka\ scheme in $L[G]$, as desired. This completes the proof of theorem~\ref{th:DefinableORDnotMahlo}.
\end{proof}

\begin{corollary}
  If $0^\sharp$ exists, then there is a class-forcing extension $L[G]$ of the constructible universe in which the generic \Vopenka\ principle holds, but there are no $\Sigma_2$-reflecting cardinals and hence no remarkable cardinals.
\end{corollary}

\begin{proof}
Consider the model $L[G]$ constructed in theorem~\ref{th:DefinableORDnotMahlo}. In that model, we have the generic \Vopenka\ principle, yet there is a parameter-free $\Delta_2$-definable class club $C$ containing no regular cardinals. Since the definition of $C$ uses no parameters, it follows that if $\kappa$ were $\Sigma_2$-reflecting, then $C\intersect\kappa$ would be unbounded in $\kappa$ and consequently $\kappa\in C$, contrary to the assumption that there are no regular cardinals in $C$. So there can be no $\Sigma_2$-reflecting cardinals. And since every remarkable cardinal is $\Sigma_2$-reflecting, it follows similarly that there can be no remarkable cardinals in $L[G]$.
\end{proof}

This provides the negative answer to question~\ref{Question:gVS-implies-remarkables?}.

\begin{corollary}\label{Corollary.GVP-with-weak-but-not-strong}
 It is relatively consistent with \GBC\ that every class $A$ admits a (weakly) virtually $A$-extendible cardinal (and so the generic \Vopenka\ principle holds), but no class $A$ admits a (strongly) virtually $A$-extendible cardinal.
\end{corollary}

\begin{proof}
If $\kappa$ is (strongly) virtually extendible, then $\kappa$ is clearly $\Sigma_2$-reflecting, since the targets $j(\kappa)$ can be chosen as high in $V$ as desired, thereby capturing any witness of a $\Sigma_2$ assertion. In this case, $\kappa$ would be $\Sigma_2$-reflecting, contrary to the existence of a parameter-free $\Delta_2$-definable club containing no regular cardinals in the model of theorem \ref{th:DefinableORDnotMahlo}.
\end{proof}

\section{Can we weaken or eliminate the $0^\sharp$ assumption?}\label{sec:Eliminate-zero-sharp}

In our main theorem, we had assumed the existence of $0^\sharp$ in $V$ in order to attain the generic \Vopenka\ principle and scheme in a class-forcing extension $L[G]$ of the constructible universe, in which the ordinals were not $\Delta_2$-Mahlo. In our model, there was a parameter-free $\Delta_2$-definable class club $C$ containing no regular cardinals. It follows that there can be no $\Sigma_2$-reflecting cardinals and therefore also no remarkable cardinals, since every remarkable cardinal is $\Sigma_2$-reflecting.

It is natural to inquire whether our use of $0^\sharp$ can be weakened or eliminated. Perhaps it is natural for one to hope to prove that if the generic \Vopenka\ principle holds in $L$, then it continues to hold in our forcing extension $L[G]$. Doing so would not only improve the theorem, by weakening the hypotheses, but it would also address the possibly unnecessary meta-mathematical aspect of the argument, whereby we assume $0^\sharp$ in $V$, but then force over $L$.

But alas, if these hypotheses are consistent, then that will not be possible. The reason is that the generic \Vopenka\ principle, if consistent, has a strictly weaker consistency strength than the theory we obtain in $L[G]$, namely, the generic \Vopenka\ principle plus $\ORD$ is not $\Delta_2$-Mahlo, which implies that there are no $\Sigma_2$-reflecting cardinals and therefore no remarkable cardinals. By the bifurcation result of Bagaria, Gitman and Schindler \cite{BagariaGitmanSchindler:VopenkaPrinciple}, since there are no remarkable cardinals in our model, then there must be a proper class of virtually rank-into-rank cardinals, and the consistency strength of this is strictly higher than the generic \Vopenka\ principle itself (but less than $0^\sharp$).

We find it quite reasonable to expect to prove our theorem starting from a model with a suitable proper class of virtually rank-into-rank cardinals, replacing our indiscernibility embeddings with those virtual rank-into-rank embeddings. We leave the details of this to another project.

Meanwhile, it is easy to see that the existence of $0^\sharp$ has a strictly higher consistency strength than is necessary for our conclusion, where we have the generic \Vopenka\ principle with a $\Delta_2$-definable class club. The reason is that our conclusion is expressible in the first-order language of set theory and therefore reflects to the initial segments of $L$, cut off at any Silver indiscernible ordinal. So under $0^\sharp$, we are able to construct transitive models of our target theory, and so $0^\sharp$ is strictly stronger than necessary.

\section{Tying up a loose end}

We'd like to conclude our paper by tying up a certain loose end, by proving that some various large cardinals properties that have been considered in the literature in connection with the \Vopenka\ scheme are actually equivalent.

Recall that a cardinal $\kappa$ is \emph{$\Sigma_n$-correct}, if $V_\kappa\prec_{\Sigma_n}V$. Let $C^{(n)}$ be the class of $\Sigma_n$-correct cardinals. Bagaria~\cite{Bagaria:CnCardinals} proved that $\VP(\Sigma_{n+2})$ holds precisely when there is a proper class of $C^{(n)}$-extendible cardinals, where a cardinal $\kappa$ is $C^{(n)}$-extendible (in Bagaria's sense), if for every $\lambda>\kappa$, there is an elementary embedding $j:V_\lambda\to V_\theta$ with critical point $\kappa$ and $\lambda<j(\kappa)\in C^{(n)}$. It is easy to see that every extendible cardinal $\kappa$ is $C^{(1)}$-extendible, because the $j(\kappa)$ of an extendibility embedding $j$ is inaccessible.

An observant reader will have noticed, however, that there is a possible collision in the terminology, since we have two possibly different conceptions of what it means to be $C^{(n)}$-extendible. Namely, on the one hand, we have the notion of $C^{(n)}$-extendible in the sense of Bagaria, which we just defined in the previous paragraph. On the other hand, we have the notion of $C^{(n)}$-extendible in the sense defined in the introduction of this article, that is, $A$-extendible when $A$ happens to be the class $C^{(n)}$. There is also another concept of relative extendibility, due to Bagaria, namely, $\kappa$ is \emph{$C^{(n)+}$-extendible}, if for every $\lambda\in C^{(n)}$ above $\kappa$ there is $\theta\in C^{(n)}$ and an elementary embedding $j:V_\lambda\to V_\theta$ with critical point $\kappa$ and $\lambda<j(\kappa)$. Bagaria proved that the least $C^{(n)}$-extendible cardinal is also $C^{(n)+}$-extendible \cite{Bagaria:CnCardinals}. And recall from section~\ref{sec:LC-characterizations} our definition that $\kappa$ is \emph{$(\Sigma_n)$-extendible}, if it is $A$-extendible when $A$ is the $\Sigma_n$-truth predicate.

We shall now happily prove that all these notions coincide, and so there is actually no collision in the terminology after all!

\begin{theorem}The following are equivalent for any cardinal $\kappa$ and any particular finite $n\geq 1$.
\begin{enumerate}
  \item $\kappa$ is $C^{(n)}$-extendible in the sense of Bagaria, so that for every $\lambda>\kappa$, there is an elementary embedding $j:V_\lambda\to V_\theta$ with critical point $\kappa$ and $\lambda<j(\kappa)\in C^{(n)}$.
  \item $\kappa$ is $C^{(n)}$-extendible, that is, $A$-extendible where $A$ is the class $C^{(n)}$, so that for every $\lambda$ there is an elementary embedding $j:\<V_\lambda,\in,C^{(n)}\intersect V_\lambda>\to\<V_\theta,\in,C^{(n)}\intersect V_\theta>$ with critical point $\kappa$ and $\lambda<j(\kappa)$.
  \item $\kappa$ is $A$-extendible for every $\Sigma_n$-definable class $A$, allowing parameters in $V_\kappa$.
  \item $\kappa$ is $(\Sigma_n)$-extendible, that is, $A$-extendible where $A$ is a $\Sigma_n$-truth predicate.
  \item $\kappa$ is $C^{(n)+}$-extendible in the sense of Bagaria, so that for every $\lambda\in C^{(n)}$ above $\kappa$, there is an elementary embedding $j:V_\lambda\to V_\theta$ for some $\theta\in C^{(n)}$ with critical point $\kappa$ and $\lambda<j(\kappa)$.
\end{enumerate}
\end{theorem}

\begin{proof}
This is a theorem scheme, a separate theorem for each finite natural number $n$ in the meta-theory.

($5\to 4$) If $j:V_\lambda\to V_\theta$ for $\lambda,\theta\in C^{(n)}$, then both $V_\lambda$ and $V_\theta$ are correct about $\Sigma_n$-truth, and so $j$ is elementary in the language with a predicate for $\Sigma_n$-truth. In other words, $\kappa$ is $(\Sigma_n)$-extendible.

($4\to 3$) If $j:V_\lambda\to V_\theta$ is elementary with respect to the predicate for $\Sigma_n$-truth, then it is also elementary with respect to any $\Sigma_n$-definable class, allowing parameters in $V_\kappa$, since that class $A$ is definable from the $\Sigma_n$-truth predicate.

($3\to 2$) Immediate, since $C^{(n)}$ is $\Pi_n$-definable, and being elementary for the complement of a predicate is the same as being elementary for the predicate itself.

($2\to 1$) If $\kappa$ is $C^{(n)}$-extendible in the sense stated, then in particular, $\kappa$ must be a limit point of $C^{(n)}$ and hence also an element of $C^{(n)}$. So if $j:\<V_\lambda,\in,C^{(n)}\intersect V_\lambda>\to\<V_\theta,\in,C^{(n)}\intersect V_\theta>$ has critical point $\kappa$, it follows that $j(\kappa)\in C^{(n)}$, as needed for Bagaria's notion.

($1\to 5$) Assume that $\kappa$ is $C^{(n)}$-extendible in the sense of Bagaria, and consider any $\lambda\in C^{(n)}$ above $\kappa$. Since any $\Sigma_2(C^{(n)})$-assertion reflects below $\kappa$, it follows that $\kappa\in C^{(n+2)}$. Let $\bar\lambda\in C^{(n+2)}$ be larger than $\lambda$. By the extendibility assumption, we get an elementary embedding $j:V_{\bar\lambda}\to V_{\bar\theta}$ for some $\bar\theta$ with critical point $\kappa$ and $\bar\lambda<j(\kappa)\in C^{(n)}$. It follows that $V_{\bar\theta}$ is correct about $C^{(n)}$ below $j(\kappa)$, although it may be possibly wrong about $C^{(n)}$ above $j(\kappa)$. Let $\theta=j(\lambda)$, so that $j\restrict V_\lambda:V_\lambda\to V_\theta$ is an elementary embedding, and furthermore $j\restrict V_\lambda\in V_{\bar\theta}$. Since $\lambda\in C^{(n)}$ and $V_{\bar\lambda}$ knows this, it follows that $\theta=j(\lambda)$ is in $(C^{(n)})^{V_{\bar\theta}}$, even though $V_{\bar\theta}$ may disagree with $V$ about $C^{(n)}$. So $V_{\bar\theta}$ thinks that ``there is $\theta\in C^{(n)}$ and an elementary embedding $h$ from $V_\lambda$ to $V_\theta$, with critical point $\kappa$ and $\lambda<h(\kappa)$.'' Since the class $C^{(n)}$ is $\Pi_n$-definable, this is a $\Sigma_{n+1}$-expressible statement about $\kappa$ and $\lambda$, which is true in $V_{\bar\theta}$. Since $V_{\bar\lambda}$ knows that $\kappa\in C^{(n+2)}$, it follows that $V_{\bar\theta}$ thinks that $j(\kappa)$ is in $C^{(n+2)}$---although it could be wrong about this---and so the $\Sigma_{n+1}$ statement about $\kappa$ and $\lambda$ reflects from $V_{\bar\theta}$ to $V_{j(\kappa)}$. So $V_{j(\kappa)}$ thinks there is an ordinal $\theta'\in C^{(n)}$ and elementary embedding $h:V_\lambda\to V_{\theta'}$ with critical point $\kappa$ and $\lambda<h(\kappa)$. Since $V_{j(\kappa)}$ is right about $C^{(n)}$, this verifies that $\kappa$ is $C^{(n)+}$-extendible in $V$, as desired.
\end{proof}

Bagaria and Andrew Brooke-Taylor had previously shown that  every $C^{(n)}$-extendible cardinal is either $C^{(n)+}$-extendible or a limit of $C^{(n)+}$-extendibles \cite{BagariaBrooke-Taylor:OnColimitsElementaryEmbeddings}. We found out after submitting the paper that Tsaprounis had independently shown that every $C^{(n)}$-extendible cardinal is $C^{(n)+}$-extendible \cite{Tsaprounis:OnC-n-ExtendibleCardinals}.

\bibliography{gVP,HamkinsBiblio,MathBiblio}

\begin{thebibliography}{BHTU16}

\bibitem[Bag12]{Bagaria:CnCardinals}
Joan Bagaria.
\newblock {$C^{(n)}$}-cardinals.
\newblock {\em Arch. Math. Logic}, 51(3-4):213--240, 2012.

\bibitem[BBT13]{BagariaBrooke-Taylor:OnColimitsElementaryEmbeddings}
Joan Bagaria and Andrew Brooke-Taylor.
\newblock On colimits and elementary embeddings.
\newblock {\em J. Symbolic Logic}, 78(2):562--578, 2013.

\bibitem[BGS17]{BagariaGitmanSchindler:VopenkaPrinciple}
Joan Bagaria, Victoria Gitman, and Ralf Schindler.
\newblock Generic {V}op\v enka's {P}rinciple, remarkable cardinals, and the
  weak {P}roper {F}orcing {A}xiom.
\newblock {\em Arch. Math. Logic}, 56(1-2):1--20, 2017.

\bibitem[BHTU16]{BagariaHamkinsTsaprounisUsuba2016:SuperstrongAndOtherLargeCardinalsAreNeverLaverIndestructible}
Joan Bagaria, Joel~David Hamkins, Konstantinos Tsaprounis, and Toshimichi
  Usuba.
\newblock Superstrong and other large cardinals are never {L}aver
  indestructible.
\newblock {\em Arch. Math. Logic}, 55(1-2):19--35, 2016.
\newblock special volume in memory of R.~Laver.

\bibitem[For10]{Foreman:IdealsGenericElementaryEmbeddings}
Matthew Foreman.
\newblock Ideals and generic elementary embeddings.
\newblock In {\em Handbook of set theory. {V}ols. 1, 2, 3}, pages 885--1147.
  Springer, Dordrecht, 2010.

\bibitem[GHK]{GitmanHamkinsKaragila:KM-set-theory-does-not-prove-the-class-Fodor-theorem}
Victoria Gitman, Joel~David Hamkins, and Asaf Karagila.
\newblock Kelley-morse set theory does not prove the class {F}odor theorem.
\newblock in preparation.

\bibitem[GS]{GitmanSchindler:virtualCardinals}
Victoria Gitman and Ralf Schindler.
\newblock Virtual large cardinals.
\newblock To appear in the {P}roceedings of the Logic Colloquium 2015.

\bibitem[Ham]{Hamkins:The-Vopenka-principle-is-inequivalent-to-but-conservative-over-the-Vopenka-scheme}
Joel~David Hamkins.
\newblock The {Vop\v{e}nka} principle is inequivalent to but conservative over
  the {Vop\v{e}nka} scheme.
\newblock manuscript under review.

\bibitem[Sch01]{schindler:remarkable2}
Ralf-Dieter Schindler.
\newblock Proper forcing and remarkable cardinals. {II}.
\newblock {\em J. Symbolic Logic}, 66(3):1481--1492, 2001.

\bibitem[Sch14]{Schindler:RemarkableCardinals}
Ralf Schindler.
\newblock Remarkable cardinals.
\newblock In {\em Infinity, computability, and metamathematics}, volume~23 of
  {\em Tributes}, pages 299--308. Coll. Publ., London, 2014.

\bibitem[SRK78]{SolovayReinhardtKanamori1978:Strong-axioms-of-infinity-and-elementary-embeddings}
Robert~M. Solovay, William~N. Reinhardt, and Akihiro Kanamori.
\newblock Strong axioms of infinity and elementary embeddings.
\newblock {\em Ann. Math. Logic}, 13(1):73--116, 1978.

\bibitem[Tsa]{Tsaprounis:OnC-n-ExtendibleCardinals}
Kostas Tsaprounis.
\newblock On {$C^{(n)}$}-extendible cardinals.
\newblock Submitted.

\end{thebibliography}
\bibliographystyle{alpha}
\end{document}